\documentclass[10pt]{amsart}

\usepackage{amsmath,amsthm,amscd,amsfonts,amssymb,epic,eepic,bbm,enumitem,graphicx}
\usepackage[pagebackref,colorlinks=true,linkcolor=blue,citecolor=blue]{hyperref}
\usepackage{amsrefs}

\newtheorem{thm}{Theorem}[section]

\newtheorem{lem}[thm]{Lemma}

\theoremstyle{remark}
\newtheorem*{rem}{Remark}

\newcounter{remarkscounter}

\numberwithin{equation}{section}

\newcommand{\abs}[1]{\lvert #1 \rvert}
\newcommand{\norm}[1]{\lVert #1 \rVert}

\newcommand{\Or}{\operatorname{O}}

\newcommand{\NS}{\operatorname{NS}}

\newcommand{\lto}{\longrightarrow}

\newcommand{\lito}{\lhook\joinrel\longrightarrow}

\theoremstyle{definition}
\newtheorem{defn}[thm]{Definition}

\renewcommand{\bar}{\overline}

\numberwithin{equation}{subsection}

\title{Extended Klein model and a bound on curves with negative self-intersection}
\author{Xin Xiong}

\begin{document}

\begin{abstract}
Let $S$ be a reduced irreducible smooth projective surface over an algebraically closed field and $\mathcal{F}$ a collection of reduced irreducible
curves with negative self-intersection on $S$ such that no positive combination $aC_1 + bC_2$ is in the divisor class of a connected nef divisor.
We denote the Picard number of $S$ by $\rho(S)$. Chinburg and Stover showed that $\abs{\mathcal{F}} \leq R_{\rho(S) - 1}(\pi/2)$ where $R_n(\pi/2)$
is the strict hyperbolic kissing number. In this paper, we use an extension of the Klein model of hyperbolic space to show that
$\abs{\mathcal{F}} < uv^{\rho(S)}$ for some positive real numbers $u, v$.
\end{abstract}

\maketitle

\section{Introduction}
Let us call an irreducible reduced curve with negative self-intersection a negative curve. We will prove the following theorem.
\begin{thm} \label{thm1}
Let $S$ be an irreducible smooth projective surface over an algebraically closed field and $\mathcal{F}$ a collection of negative curves. If for $C_1, C_2 \in \mathcal{F}$ and
positive integers $a, b$, $aC_1 + bC_2$ is not in the divisor class of a connected nef divisor then $\abs{\mathcal{F}} < uv^{\rho(S)}$
for positive real numbers $u, v$ independent of $S$.
\end{thm}

In a recent paper, Chinburg and Stover showed using hyperbolic codes that for sufficiently large $\rho(X)$,
$\abs{\mathcal{F}} \leq R_{\rho(S) - 1}(\pi/2)$ where $R_k(\pi/2)$ denotes the strict hyperbolic kissing number in dimension $k$ and angular separation of $\pi / 2$.
\cite[\S 3]{cs11}.
We will show a similar bound on $\abs{\mathcal{F}}$ using an extension of the Klein disc model for hyperbolic space.

From the N\'{e}ron-Severi theorem, we know that the N\'{e}ron-Severi group $\NS(S)$ is a finitely generated abelian group of rank $\rho(S)$.
We may extend the intersection pairing via the map $\NS(S) \lto \NS(S) \otimes_\mathbb{Z} \mathbb{R}$
and by the Hodge index theorem \cite{hartag}, $\NS(S) \otimes_\mathbb{Z} \mathbb{R} \cong \mathbb{R}^{1, n}$ as an inner product space where $n = \rho(S) - 1$. Denote $\mathbb{R}^{1, n}$ is $\mathbb{R}^{n + 1}$ endowed a signature $(1, n)$ inner product $H(\cdot, \cdot)$.
In other words, tensoring with $\mathbb{R}$ extends the intersection pairing to be a signature $(1, n)$ inner product which we call $H(\cdot, \cdot)$.

\section{Extended Klein model}
We now introduce a model of $(\mathbb{R}^{1, n} - \{0\}) / \mathbb{R}^+$ that easily exhibits orthogonality with respect to $H(\cdot, \cdot)$.
We will see later that for theorem \ref{thm1} we only require the sign of $H(\cdot, \cdot)$.
In our notation, let $\mathbb{R}^{n+1}$ be parametrized by coordinates $x_0, \ldots, x_n$. We first remind the reader that the Klein disc model models the points of hyperbolic
$n$-space as the disc $\mathcal{K}^n = \{ 1 \} \times D^n \subset \mathbb{R}^{n + 1}$ where $D^n$ is the open disc of radius 1 centered on the origin \cite{cannon}.
Alternatively, given the hyperboloid model, we may define the Klein disc model as the projection onto the plane $x_0 = 1$ from the origin.

\begin{rem}
We will omit discussion of the hyperbolic metric on $\mathcal{K}^n$ since we need only orthogonality to prove \ref{thm1}.
Unless stated otherwise, all metrics within this paper will default to the Euclidean metric obtained via the isomorphism $\mathbb{R}^{1, n} \cong \mathbb{R}^{n + 1}$.
Additionally, $\norm{\cdot}$ will always denote the Euclidean norm and $\norm{\cdot}_H$ the hyperbolic norm induced by $H(\cdot, \cdot)$.
\end{rem}

For this contruction, define the map $\mathfrak{s} : (\mathbb{R}^{1, n} - \{0\}) / \mathbb{R}^+ \lto \{-1, 0, 1\}$ mapping a point $x$ to the sign of its norm $H(x, x)$.
Consider the discs $\mathcal{D}^+ = \{ 1 \} \times D^n, \hspace{8pt} \mathcal{D}^- = \{-1 \} \times D^n$. Their union $\mathcal{D} = \mathcal{D}^+ \cup \mathcal{D}^-$
is a section of $\mathfrak{s}^{-1}(-1)$. Likewise we may consider $\mathcal{C} = (-1, 1) \times S^{n - 1}$, which is a section of $\mathfrak{s}^{-1}(1)$.
We see that $\partial \mathcal{D} = \partial \mathcal{C}$ is a section of $\mathfrak{s}^{-1}(0)$.

\begin{figure}
\centering
\caption{extended Klein model when $n = 2$}
\includegraphics[scale=0.3]{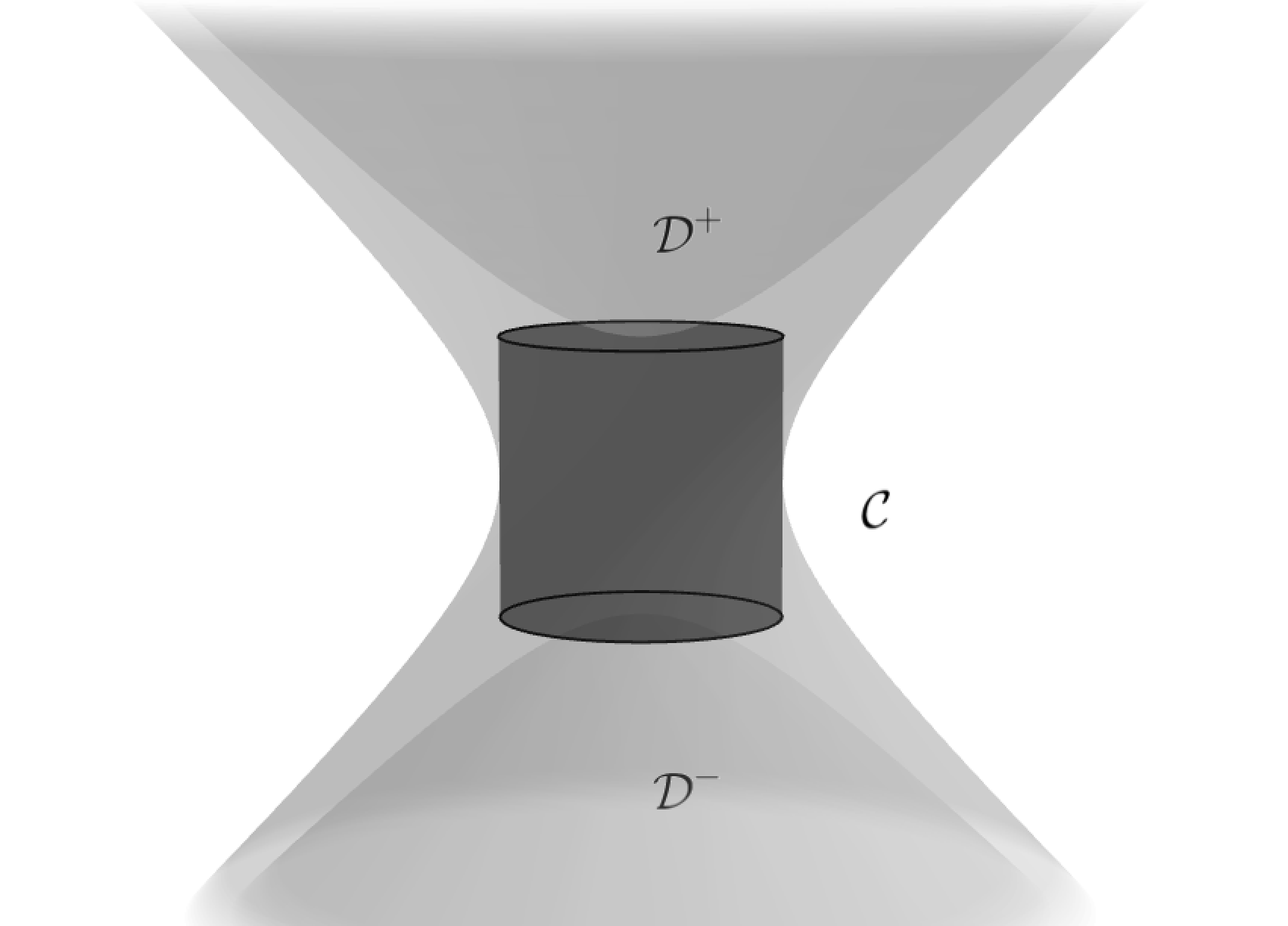}
\label{fig1}
\end{figure}

\begin{defn}[Extended Klein model]
The extended Klein model is $\mathcal{E} = \mathcal{D} \cup \mathcal{C} \cup \partial \mathcal{D}$. (See figure \ref{fig1}.)
\end{defn}

\label{constr} If we take a point $c \in \mathcal{C} \subset \mathcal{E}$, we may describe its orthogonal complement as follows.
Acting by $\Or(n) \subset \Or^+(1, n)$, we may assume that $c = (x_c, 1, 0, \ldots, 0)$.
Then the subspace $L \subset \mathbb{R}^{1, n}$ orthogonal to it intersects $\mathcal{D}$ in two disjoint $(n-1)$-discs, one within each of $\mathcal{D}^\pm$.
Without loss of generality, let us work on $\mathcal{D}^+$. Call $z$ the projection of $c$ onto $\partial \mathcal{D}^+$ and $D = L \cap \mathcal{D}^+$.
Using a symmetry argument and $H(\cdot, \cdot)$, we see that the central point of $D$ is the point on $D$ closest to $c$. We will call this point $y$.
From the coordinate representations $c = (x_c, 1, 0, \ldots, 0), z = (1, 1, 0, \ldots, 0), y = (1, x_y, 0, \ldots, 0)$ and acting by $\Or(n)$ to generalize,
we see that $\norm{y - z} = \norm{c - z}$. (See figure \ref{fig2}.)

\begin{figure}
\centering
\caption{$c$ with its orthogonal complement $L$ and its vertical projection $z$}
\includegraphics[scale=0.3]{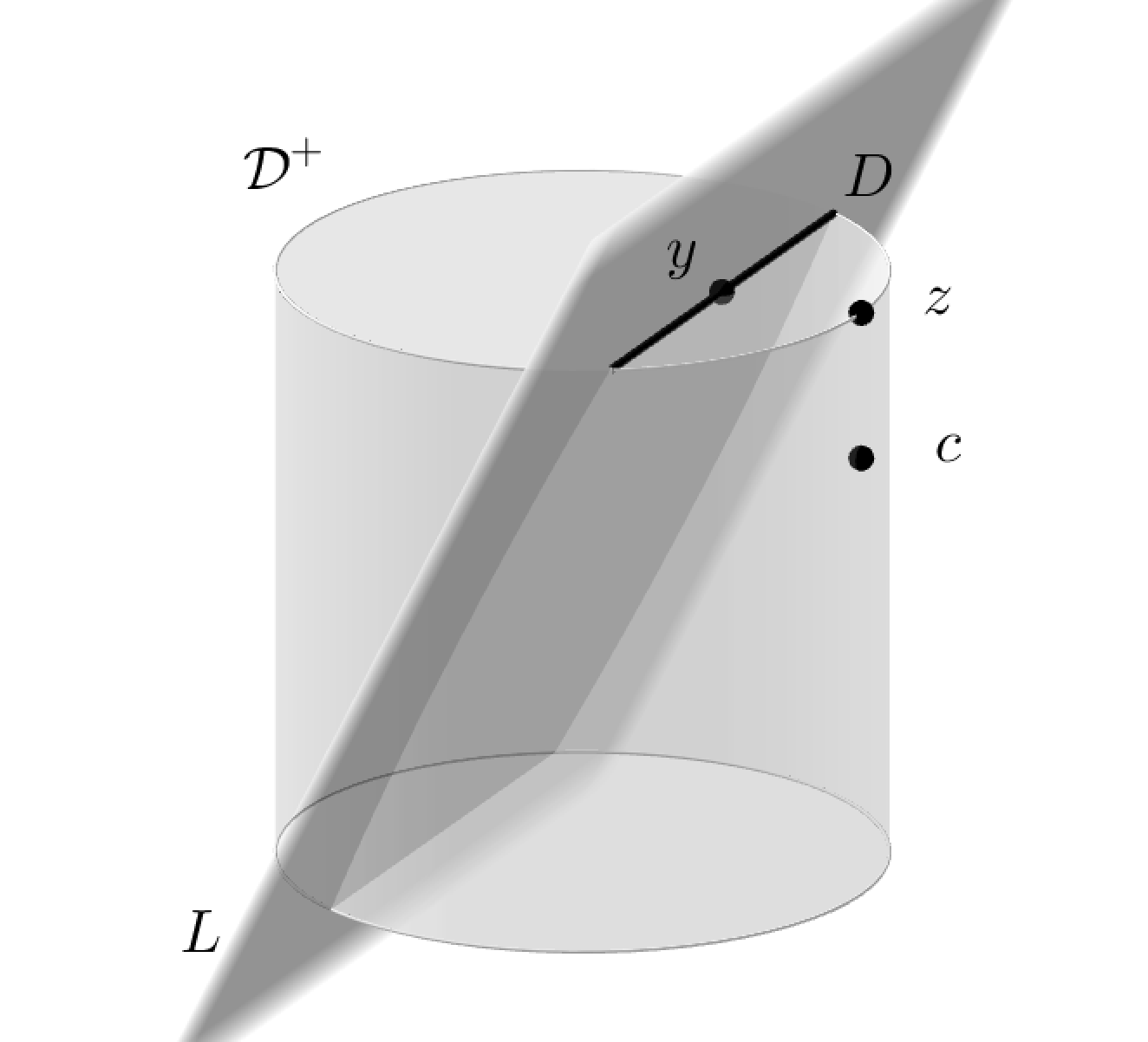}
\label{fig2}
\end{figure}

\begin{lem}
\label{lem2}
We may represent a point $c \in \mathcal{C}$ uniquely by $(z, \theta) \in S^{n - 1} \times (0, \pi)$ with the explicit correspondence given in the proof.
\end{lem}

\begin{proof}
With discussion and notation from the last paragraph, let $S^{n - 1}$ be identified with $\partial \mathcal{D}^+$ above and $z$ the projection of $c$ onto the aforementioned sphere.
Let $\theta$ be the angular size of the cap defined by $L \cap \partial \mathcal{D}^+$ containing $z$. Equivalently, $\theta = \arccos(1 - \norm{y - z})$.
\end{proof}

\section{Proof of Theorem \ref{thm1}}
$\mathcal{F}$ must satisfy that for any $C_i, C_j \in \mathcal{F}$
\begin{enumerate}[label=(\Roman*), ref=\Roman*]
\item $C_i^2 < 0$ \label{Cond1}
\item $C_i \cdot C_j \geq 0$ \label{Cond2}
\item $(aC_i + bC_j)^2 \leq 0 \hspace{10pt} \forall a, b \in \mathbb{N}$ \label{Cond3} .
\end{enumerate}
The first two conditions are due to $C_i, C_j$ being curves of negative self-intersection.
If condition \ref{Cond3} fails, i.e. $(aC_i + B_j)^2 > 0$, then $(aC_i + B_j)$ is connected nef.
Note that these conditions are invariant under scaling by $\mathbb{R}^+$.
Let $\Phi$ be the composition of maps $(\NS(S) - \{0\}) \lto (\NS(S) \otimes_\mathbb{Z} \mathbb{R} - \{0\}) \lto \mathcal{E}$.
We map $\mathcal{F}$ through $\Phi$ then denote $c_i = \Phi(C_i)$ and $\mathcal{G} = \Phi(\mathcal{F})$.

\begin{thm}
\label{thm2}
Let $c_i \in \mathcal{G}$ be represented by $(z_i, \theta_i)$ as in lemma \ref{lem2}.
Let $\delta_{ij}$ be the angular distance between $z_i$ and $z_j$ on $\partial \mathcal{D}^+$.
$\mathcal{G}$ must satisfy
\begin{enumerate}[label=(\roman*), ref=\roman*]
\item $c_i \in \mathcal{C}$ \label{cond1}
\item $\cos \delta_{ij} \leq \cos \theta_i \cos \theta_j$ \label{cond2}
\item $\theta_i + \theta_j \geq \delta_{ij}$. \label{cond3}
\end{enumerate}
\end{thm}

\begin{rem}
The supremum of $\abs{\mathcal{G}}$ when $\dim \mathcal{E} = n$ is the hyperbolic kissing number $\bar{R}_n(\pi / 2)$. This number is at least $R_n(\pi / 2)$ \cite{cs11}.
\end{rem}

\begin{proof}
We will show that the conditions denoted by the same roman numerals in theorem \ref{thm2} and the discussion above it are equivalent.
Conditions \ref{Cond1} and \ref{cond1} are equivalent since $\mathcal{C}$ contains precisely the elements of $c \in \mathcal{E}$ such that $\norm{c}_H < 0$.

Condition \ref{Cond2} is equivalent to $H(c_i, c_j) \geq 0$. With notation from page \pageref{constr}, we see that $L_i$ bisects $\mathcal{E}$ into two connected components.
By bicontinuity of the inner product, the component containing $z_i$ contains all points $x$ such that $H(c_i, x) < 0$.
Acting by $\Or(n)$, we may without loss of generality assume $c_i = (\cos \theta_i, 1, 0, \ldots, 0), \ c_j = (\cos \theta_j, \cos \delta_{ij}, \sin \delta_{ij}, 0, \ldots, 0)$.
Recall that $L_i \cap \mathcal{D} = \{ (\pm 1, \pm \cos \theta_i, x_2, \ldots, x_n) \}$, so $L_i$ intersects $x_0 = \cos \theta_2$ at
$\{ (\cos \theta_j, \cos \delta_{ij} \sin \delta_{ij}, x_2, \ldots, x_n) \}$. We then argue that $\cos \delta_{ij}$, the $x_1$ coordinate of $c_j$ not greater than
$\cos \theta_i \cos \theta_j$, which is condition \ref{cond2}. (See figure \ref{fig3}.)

\begin{figure}
\centering
\caption{projection of $c_j$, $L_i$, and the hyperplane $x_0 = \cos \theta_j$ onto the $x_0x_1$ plane}
\includegraphics[scale=0.3]{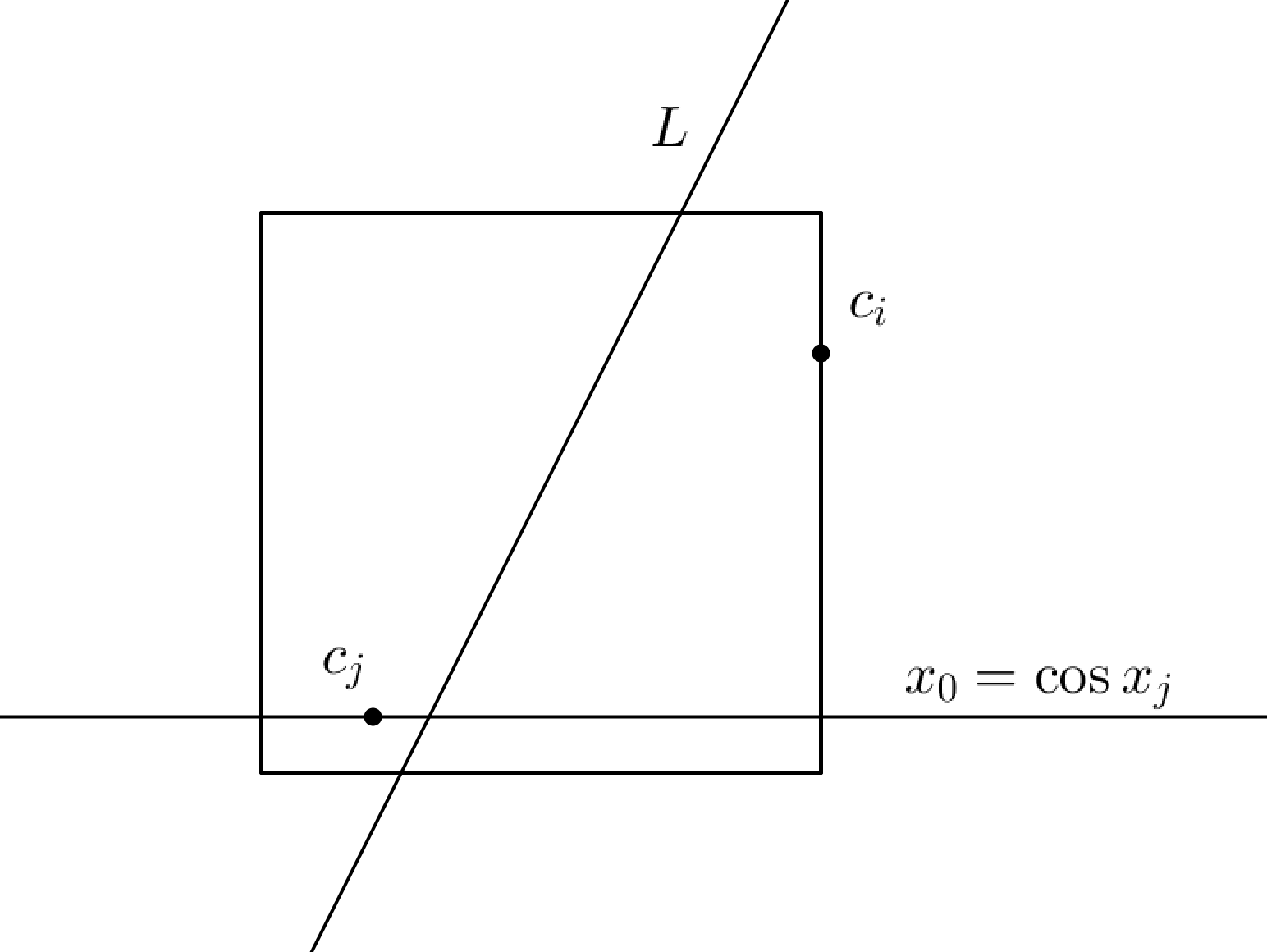}
\label{fig3}
\end{figure}

Now we turn to show conditions \ref{Cond3} and \ref{cond3} are equivalent. First we note that \ref{cond3} is equivalent to
$\bar{D_i} \cap \bar{D_j} \neq \emptyset$ with equality case $\dim (\bar{D_i} \cap \bar{D_j}) < \dim D_i$, i.e. when the interior $D_i, D_j$ do not intersect.
Passing through the map $\Phi$, we see that \ref{Cond3} is equivalent to $\norm{ac_i + c_j}_H^2 \leq 0$ for all $a \in \mathbb{R}^+$.
Since this norm is invariant under $\Or^+(1, n)$, we may act by isometry and assume without loss of generality that
$c_i = (0, 1, 0, \ldots, 0)$ and $c_j = (\cos \theta_j, \cos \delta_{ij}, \sin \delta_{ij}, 0, \ldots, 0)$.
Maximizing with respect to $a$ and differentiating, we find that $a = - \cos \delta_{ij}$ gives the maximum value of $\norm{ac_i + c_j}_H^2 = cos^2 \theta_j - \sin^2 \delta_{ij}$.
Thus we have condition \ref{Cond3} equivalent to $\abs{\cos \theta_j} \leq \sin \delta_{ij}$ when $c_i = (0, 1, 0, \ldots, 0)$.
From here, it suffices to show equivalence of $\abs{\cos \theta_j} = \sin \delta_{ij}$ and $\dim (\bar{D_i} \cap \bar{D_j}) < \dim D_i$ with $n = 2$ since intersection,
$\mathcal{D}^+$, and $\norm{\cdot}_H$ are all invariant under action by $\Or^+(1, n)$ and we may vary $\theta_j$ slightly to obtain the equivalence of \ref{Cond3} and \ref{cond3}.

\begin{figure}
\centering
\caption{construction to prove the equivalence of \ref{Cond3} and \ref{cond3}}
\includegraphics[scale=0.3]{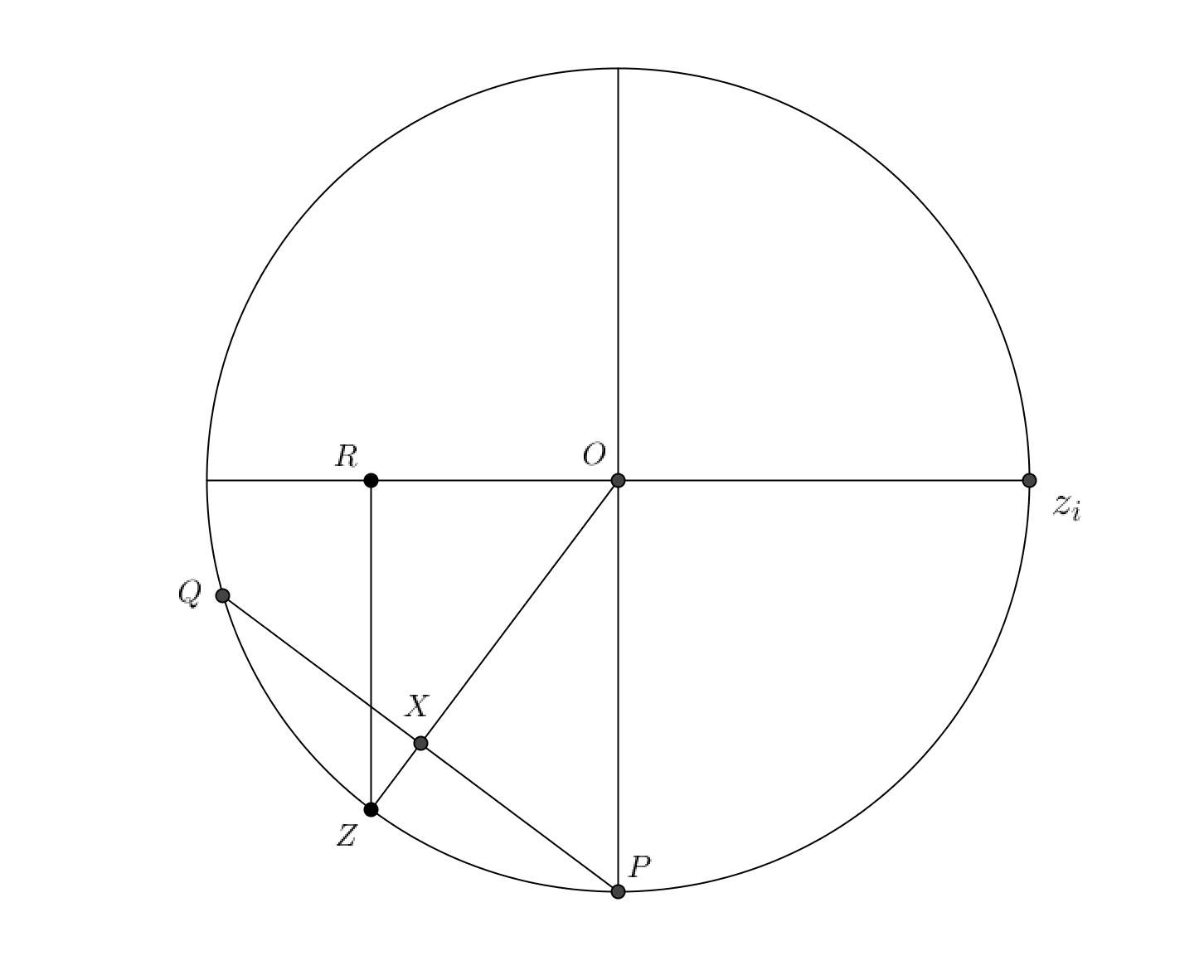}
\label{fig4}
\end{figure}

From here, we parametrize $\mathcal{D}^+ = \{ (x_1, x_2) | x_1 + x_2 \leq 1 \}$. Recall that $z_i = (1, 0)$. (See figure \ref{fig4} for the following construction.)
Let $O = (0, 0), P = (0, 1), Z = z_j, R = (\cos \delta_{ij}, 0)$. Let $Q \in \partial \mathcal{D}^+$ such that $\bar{OZ}$ bisects $\bar{PQ}$ and let $X$ be their intersection.
We then have $\triangle ROZ \equiv \triangle XPO$ and a fortiori $\abs{\bar{QZ}} = \abs{\bar{XO}}$.
Letting $\bar{D_j} = \bar{PQ}$ gives $\cos_j = \sin \delta_{ij}$ and $\theta_i + \theta_j = \delta{ij}$.
Small variations on $\theta_j$ show the equivalence of the inequalities \ref{Cond3} and \ref{cond3}.
\end{proof}

\begin{proof}[Proof of Theorem \ref{thm1}] 
First, we can reduce condition \ref{cond2} to
\begin{enumerate}[label=(\roman**),ref=\roman**,start=2]
\item $\theta_i < \delta_{ij}$. \label{cond2*}
\end{enumerate}
Next, we require that $c_i \geq 0$ for all $c_i \in \mathcal{G}$ by throwing out at most half of $\mathcal{G}$.
Embedding $\partial \mathcal{D}^+ \lito \mathbb{R}^n$, we may reduce \ref{cond2*} and \ref{cond3} to bounding a collection $\mathcal{B}$ of open balls $B_i = B(z_i, \theta_i)$ of radii $\theta_i$
centered at $z_i$ in Euclidean space such that $z_i \notin B_j$ and $B_i \cap B_j \neq \emptyset$ for any balls $B_i, B_j \in \mathcal{B}$. Consistent with prior notation, let $\delta_{ij} = \norm{z_i - z_j}$.

A priori, let $\abs{\mathcal{B}}$ be finite. Choose $B_0, B_1$ and rescale the ambient space such that $\delta_{01} = \min \delta_{ij} = 1$.
Let $\mathcal{R}^- = \{ p \in \mathbb{R}^n | \norm{p - z_0} < 2 \}$ and $\mathcal{R}^+ = \{ p \in \mathbb{R}^n | \norm{p - z_0} \geq 2 \}$
and $\mathcal{B}^\pm = \{ B_i \in \mathcal{B} | z_i \in \mathcal{R}^\pm \}$. We have a naive bound $\abs{\mathcal{B}^-} < 2^{n+1}$.
Let $B_i, B_j \in \mathcal{B}^+$ and introduce coordinates such that $z_0 = 0, \ z_i = (0, k, \ldots, 0)$ with $k \geq 2$, and $p_j = (x, y, 0, \ldots, 0)$.
From conditions \ref{cond2*} and \ref{cond3}, we have $\sqrt{x^2 + y^2} - 1 < \sqrt{x^2 + (y - k)^2}$, which along with $x^2 + y^2 \geq 4$ from $B_j \in \mathcal{B}^+$
tells us that $B_i$ restricts $B_j$ outside of a cone of angle at least $2\arctan \left( \frac{\sqrt{15}}{7} \right)$.
Thus, $\abs{\mathcal{B}^+}$ is bounded by some exponential function of $n$ and $\abs{\mathcal{F}}$ bounded by an exponential function of $\rho(S)$.
\end{proof}

\section*{Acknowledgements}
I would like to thank De-Qi Zhang for pointing out that $aC_1 + bC_2$ in theorem \ref{thm1} not being ample is too weak a hypothesis
and Ted Chinburg for suggesting that it should be reduced to connected nef.

\bibliography{bibl}
\end{document}